\documentclass[12pt]{amsart}
\usepackage{amsfonts}
\usepackage{amsmath}
\usepackage{amsthm}
\usepackage{amssymb}
\usepackage{graphicx}
\usepackage{amscd}
\usepackage{xcolor}
\usepackage{tikz-cd}
\usepackage{hyperref}
\usepackage[utf8]{inputenc}
\usepackage[greek,english]{babel}
\usepackage{alphabeta}
\usepackage{fancyhdr}
\usepackage{mathtools}
\usepackage[toc,page]{appendix}
\newtheorem{theorem}{Theorem}

\newtheorem{corollary}[theorem]{Corollary}

\newtheorem{lemma}[theorem]{Lemma}

\newtheorem{remark}{Remark}

\begin{document}
\title[Estimates of the derivatives of the heat kernel ]{The
derivatives of the heat kernel on symmetric spaces}
\dedicatory{To the memory of Professor Michel Marias.}

\author{Anestis Fotiadis}
\email{fotiadisanestis@math.auth.gr}
\author{Effie Papageorgiou}
\email{papageoe@math.auth.gr}
\curraddr{Department of Mathematics, Aristotle University of Thessaloniki,
Thessaloniki 54.124, Greece}
\date{}

\begin{abstract}
We derive estimates of the derivatives of the heat kernel on noncompact 
symmetric spaces and on locally symmetric spaces. Applying these
estimates we study the $L^{p}$-boundedness  of Littlewood-Paley-Stein
operators and the Laplacian of the heat operator on a wide class of locally symmetric spaces.
\end{abstract}

\maketitle

\section{Introduction and statement of the results}\label{sectionintroduction}
Our main objective in this article is to prove estimates of the derivatives of the heat kernel on noncompact symmetric spaces. We then obtain a variety of applications.

More than the results themselves, it is the method in the proof of our main result that is nontrivial. More specifically, we are able to estimate time derivatives of the heat kernel by combining sharp heat kernel estimates with rough estimates of its time derivatives, and by improving step by step the resulting estimates using an iterative argument. The final estimates obtained this way are precise.

A symmetric space is a homogeneous space that can be described as a coset
Riemannian manifold $X=G/K,$ where $G$ is a semisimple Lie group and $K$ is
a maximal compact subgroup. From now on, $%
X $ will denote an $n$-dimensional symmetric space.

Let $\mathfrak{g}$ and $\mathfrak{k}$ be the Lie algebras of $G$ and $K$
respectively. Let also $\mathfrak{p}$ be the subspace of $\mathfrak{g}$
which is orthogonal to $\mathfrak{k}$ with respect to the Killing form. Let $%
\mathfrak{a}$ be an abelian maximal subspace of $\mathfrak{p}$, $\mathfrak{a}%
^{\ast }$ its dual and let $\Sigma \subset \mathfrak{a}^{\ast }$ be the root
system of ($\mathfrak{g}$,$\mathfrak{a}$). Choose a set $\Sigma ^{+}$ of
positives roots. Let $\rho $ be the half-sum of positive roots counted with
multiplicity. Let $\mathfrak{a}^{+}\subset \mathfrak{a}$ be the
corresponding positive Weyl chamber and let $\overline{\mathfrak{a}_{+}}$ be
its closure. We have the Cartan decomposition $G=K(\exp \overline{\mathfrak{a%
}_{+}})K$. If $x\in G$, then it is uniquely written as $x=k_{1}\left( \exp H\right)
k_{2},$ with $k_{1},k_{2}\in K$ and $H\in \overline{\mathfrak{a}_{+}}$.

Let $\Delta $ be the Laplace-Beltrami operator on $X$. Then, the heat kernel 
$h_{t}$ of $X$ is the fundamental solution of the heat equation $\partial
_{t}h_{t}=\Delta h_{t}$. Note that the heat kernel is a $K$-bi-invariant function, i.e., if $%
x=k_{1}\left( \exp H\right) k_{2}\in X$, then $h_{t}(x)=h_{t}(\exp H)$.

Our main result is the following theorem.

\begin{theorem}
\label{Main result} If $X$ is a symmetric space of noncompact type, then for all $\epsilon\in(0,1)$ and $i\in \mathbb{N}$ there is a constant $c>0$ such that 
\begin{equation}  \label{required estimate}
\left\vert \frac{\partial ^{i}h_{t}}{\partial t^{i}}(\exp H)\right\vert \leq
ct^{-(n/2)-i}e^{-(1-\epsilon )\left( \left\Vert\rho \right\Vert^{2}t+\langle
\rho ,H\rangle +\left\Vert H\right\Vert ^{2}/(4t)\right)},
\end{equation}
$\text{ for all } t>0 \text{ and }H \in\overline{\mathfrak{a}_{+}}.$
\end{theorem}

There is a very rich and long literature concerning heat kernel estimates in various geometric contexts. See for example \cite{ANJ,ANO,DAVMAN,LIYAU}, and the references therein. In particular, Davies and Mandouvalos in \cite{DAVMAN}, obtained optimal estimates of the
heat kernel in hyperbolic spaces and Anker and Ji in \cite%
{ANJ} and Anker and Ostellari in \cite{ANO},
obtained estimates of the heat kernel in the case of symmetric spaces.
 Estimates of the time derivatives of the heat kernel have
been obtained in \cite{MANTSE} for hyperbolic spaces, and in \cite{DAV} and \cite{GR2}
on general geometric contexts. 

Next, we shall find applications of Theorem \ref{Main result}. Firstly, we obtain gradient
estimates of $h_{t}.$

\begin{corollary}
\label{Gradient estimate} If $X$ is a symmetric space of noncompact type,
then for all $\epsilon\in(0,1)$ there is a constant $c>0,$ such that 
\begin{equation}
\left\Vert \nabla h_{t}(\exp{H})\right\Vert \leq
ct^{-(n+1)/2}e^{-(1-\epsilon )\left(\left\Vert\rho \right\Vert^{2}t+\langle
\rho ,H\rangle +\| H\|^2/(4t)\right) },
\end{equation}
$\text{ for all } t>0\text{ and } H \in\overline{\mathfrak{a}_{+}}.$
\end{corollary}

Let $\Gamma $ be a discrete torsion free subgroup of $G$. Then the locally
symmetric space $M=\Gamma \backslash X$, equipped with the projection of the
canonical Riemannian structure of $X$, becomes a Riemannian manifold. Denote
by $\Delta _{M}$ the Laplacian on $M$ and by ${d_M}$ the Riemannian distance
on $M$. We denote by $\pi :X\longrightarrow M$ the canonical projection and
for $x\in X$ we set $\tilde{x}=\pi \left( x\right) $. Denote by 
\begin{equation*}
P_{s}({x},{y})={\sum_{\gamma \in \Gamma }}e^{-sd(x,\gamma y)},\text{ \ }%
x,y\in X,
\end{equation*}%
and 
\begin{equation*}
\delta (\Gamma )=\inf \{s\in (0,\infty ):P_{s}({x},{y})<\infty \},
\end{equation*}%
the \emph{Poincar{\'{e}} series} and the \emph{critical exponent}
respectively.

Set also 
\begin{equation*}
\rho _{m}=\min_{H\in \overline{\mathfrak{a}_{+}},\text{ }\left\Vert
H\right\Vert =1}\langle \rho ,H\rangle .
\end{equation*}

Let $\delta (\Gamma )<\left\Vert \rho \right\Vert +\rho _{m}.$ Consider $\alpha _{2}\in (\delta
(\Gamma ),\left\Vert \rho \right\Vert +\rho _{m})$ and $\alpha _{1},\alpha
_{3}\in \lbrack 0,1]$ such that $\alpha _{1}\alpha _{3}\in \left[ \left(
\left( \alpha _{2}-\rho _{m}\right) /\left\Vert \rho \right\Vert \right)
^{2},1\right] $.

Next, we obtain estimates of the heat kernel on a locally symmetric space.
\begin{theorem}
\label{Locally symmetric 2} Suppose that $M=\Gamma \backslash X,$ with $%
\delta (\Gamma )<\left\Vert\rho \right\Vert+\rho _{m}$. Then, for all $\epsilon\in(0,1)$ there is a constant $c>0,$ such that 
\begin{equation}
\begin{split}
\left| \frac{\partial ^{i}h_{t}^{M}}{\partial t^{i}}(\tilde{x},\tilde{y})
\right| &\leq  \frac{c}{t^{n/2+i}}e^{-(1-\epsilon )\left( (1-\alpha_{1}
)\left\Vert\rho \right\Vert^{2}t+(\alpha_{2}-\delta(\Gamma))d_{M}(\tilde{x},%
\tilde{y})+(1-\alpha_{3} )\frac{d_{M}^{2}(\tilde{x},\tilde{y})}{4t} \right)} \\
& \times P_{\epsilon +\delta (\Gamma )}(x,y),
\end{split}%
\end{equation}
for all $t>0$ and $\tilde{x},\tilde{y} \in M$.
\end{theorem}
Observe that the above result extends the results
of Weber in \cite{WE}.

Let $\kappa $ be a $K$-bi-invariant function and let $\ast |\kappa |$ denote the
convolution operator whose kernel is $|\kappa |$. Let also $p\in [1,\infty ]$, denote by $p^{\prime }$ its conjugate and set 
\begin{equation*}
s(p)=2\min \left\{ \frac{1}{p},\frac{1}{p^{\prime }}\right\} .
\end{equation*}
 Assume that the following version of the Kunze-Stein phenomenon holds,
\begin{equation}\tag{KS}\label{kunzestein}
\Vert \ast |\kappa |\Vert _{L^{p}(M)\rightarrow L^{p}(M)}\leq c\underset{G}{%
\int }|\kappa (g)|\phi _{-i\eta _{\Gamma }}(g)^{s(p)}{dg},
\end{equation}%
where $\phi _{\lambda }$ are the elementary spherical functions, $\eta
_{\Gamma }$ is a vector of the euclidean sphere $S(0,(\left\Vert\rho
\right\Vert^{2}-\lambda _{0})^{1/2})$ of $\mathfrak{a}^{\ast }$ and $\lambda
_{0}$ is the bottom of the spectrum of the Laplacian $-\Delta _{M}$. 
For example, this is the case for $M=\Gamma \backslash G/K$, when
	(i) $\Gamma $ \textit{is a lattice}, or (ii) $G$ \textit{possesses Kazhdan's property (T)} or (iii) $\Gamma \backslash G$ \textit{is non-amenable},
see \cite{LOMAjga} for more details.

Denote by $H_{t}=e^{t\Delta_M }$ the heat semigroup on $M$. Fix $i\in \mathbb{%
N}$. Then, for all $\sigma \geq 0$ we consider as in \cite{AN1} the Littlewood-Paley-Stein operator 
\begin{equation*}
{H}_{\sigma }(f)(x)=\left( \int_{0}^{\infty }e^{2\sigma t}\left( t^{2i-1}\left\vert
\frac{\partial ^{i}}{\partial t^{i}}H_{t}f(x)\right\vert ^{2}+\|\nabla
_{x}H_{t}f(x)\|^{2}\right) \textbf{  }\right) ^{1/2}.
\end{equation*}%

Next, we apply Corollary \ref{Gradient
estimate} and Theorem \ref{Locally symmetric 2} in order to prove $L^{p}$-boundedness of various operators related
to the heat semigroup $H_t = e^{t\Delta_{M}}$, on certain locally symmetric spaces $M$.
\begin{theorem}
\label{maximal} Suppose that $M$ satisfies \eqref{kunzestein}. Then, the
operator  ${H}_{\sigma}$ is bounded on $L^{p} (M), \;
p\in (1,\infty),$ provided that 
\begin{equation}\label{sigmacondition0}
\sigma < s (p) \left(\left\Vert\rho \right\Vert - \left\Vert \eta_{\Gamma }
\right\Vert\right) \left( 2\left\Vert\rho \right\Vert - s (p)
\left(\left\Vert\rho \right\Vert - \left\Vert \eta_{\Gamma }
\right\Vert\right)\right).
\end{equation}
\end{theorem}

The Littlewood-Paley-Stein operator was first introduced and studied by Lohou%
\'{e} \cite{LO}, in the case of Riemannian manifolds of non-positive
curvature. In a variety of geometric settings it has been proved that  ${H}_{\sigma}$ is bounded on $L^{p},\; p\in (1,\infty),$
under some conditions on $\sigma$ (see for example \cite{LO}). In
particular, in the case of a symmetric space $X$, Anker in \cite{AN1} has
shown that ${H}_{\sigma}$ is bounded in $L^{p}(X)$, provided $\sigma < 4\left\Vert \rho \right\Vert^2/p p^{\prime }$. In the
case of symmetric spaces, where $\eta_{\Gamma} = 0,$ \cite{LOMAjga}, we observe that the condition (\ref{sigmacondition0}) on $\sigma$ becomes $\sigma <
\left\Vert \rho \right\Vert^2 s (p) (2-s(p))=4 \left\Vert \rho
\right\Vert^2/ p p^{\prime },$ thus Theorem \ref{maximal} extends the result of Anker in \cite%
{AN1}. 

Finally, we obtain $L^{p}$-estimates for the operator $\Delta_{M}
e^{t\Delta_{M}}$.

\begin{theorem}
\label{Delta} Suppose that $M$ satisfies \eqref{kunzestein}. Then, for all $%
p\in [1,\infty]$ there exists $\epsilon >0$ such that 
\begin{equation}
\| \Delta_{M} e^ {- \Delta_{M} t} \|_{L^{p}(M) \rightarrow L^{p}(M)} \leq 
\begin{cases}
c t^{-1}, \quad \text{if} \quad 0 \leq t < 1, \\ 
c e^{-\epsilon t}, \quad \text{if} \quad t \geq 1,%
\end{cases}%
\end{equation}
\end{theorem}

This operator has been studied for complete Riemannian manifolds with
bounded geometry by Davies in \cite{DAV}. In \cite{DAV} the $L^{p}$-norm of the operator $%
\Delta_{M} e^{-t\Delta_{M}}$ is proved to
be bounded by a constant for $t \geq 1.$ In our case, we prove that it
decays exponentially as $t\rightarrow \infty,$ thus we extend the result of
Davies.

Let us now outline the organization of the paper. In Section \ref%
{Preliminaries} we recall some basic definitions about symmetric spaces and
the heat kernel. In addition, we recall some results providing estimates of
the heat kernel and estimates of its derivatives. In Section \ref{sectionsym} we prove Theorem \ref{Main result}. Next, in Section \ref%
{applications}, we obtain a variety of applications. Firstly, as a direct
application we prove gradient estimates of the heat kernel. Next, we prove estimates of the derivatives of the heat
kernel for locally symmetric spaces. Finally, we study Littlewood-Paley-Stein operators and the Laplacian of the heat operator. We prove
that they are bounded on $L^{p}(M),$ where $M$ is a locally symmetric space satisfying (\ref{kunzestein}).

Throughout this article the different constants will always be denoted by
the same letter $c$. When their dependence or independence is significant,
it will be clearly stated.

\section{Preliminaries}\label{Preliminaries}

\subsection{Symmetric and locally symmetric spaces}

\label{The heat operator on symmetric and locally symmetric spaces} We shall
recall some basic facts on symmetric and locally symmetric spaces. For
details, see \cite{AN1,LOMAjga}.

As it is already mentioned in the Introduction, a symmetric space $X$ is the Riemannian manifold  $G/K$ where $G$ is a  real semisimple Lie group, connected, noncompact, with finite center  and $K$ is a maximal
compact subgroup of $G$.

Let $\mathfrak{g}$ be the Lie algebra of $G$ and $\mathfrak{k}$ the Lie
algebra of $K$ respectively. Denote by $\mathfrak{p}$ the orthogonal
complement of $\mathfrak{k},$ then $\mathfrak{g}= \mathfrak{k} \oplus 
\mathfrak{p}, $ is the Cartan decomposition at the Lie algebra level. Let us
choose a maximal abelian subspace $\mathfrak{a}$ of $\mathfrak{p}$. Let $%
\mathfrak{a}^{\ast}$ be the dual space of $\mathfrak{a}$. For any $\alpha
\in \mathfrak{a}^{\ast}$, let $ad(X)(Y)=[X,Y]$ for all $X,Y\in \mathfrak{g}$
and set 
\begin{equation*}
\mathfrak{g}_\alpha : =\{Y \in \mathfrak{g}: \quad {ad} (H) (Y) = \alpha (H)
Y, \quad \text{ for all } H \in \mathfrak{a}\}.
\end{equation*}
If $\mathfrak{g}_\alpha \neq \{ 0 \},$ then $\alpha \in \mathfrak{a}^{\ast}
\backslash\{0\}$ is called a \emph{root} of the pair $(\mathfrak{g},%
\mathfrak{a})$ and $\mathfrak{g}_\alpha$ is called the root space. Denote by 
$\Sigma $ the set of all roots. If $\alpha$ is a root, then the only
multiples of $\alpha$ that can also be roots are $\pm \frac{1}{2} \alpha,\
\pm 2 \alpha$, and $-\alpha.$ A positive root $\alpha$ is called
indivisible if $\frac{1}{2} \alpha$ is not a root. We call an $H \in 
\mathfrak{a}$ \emph{regular} if $\alpha (H) \neq 0$ for all $\alpha \in
\Sigma$. The set of regular elements of $\mathfrak{a}^{\ast}$ is the
complement of a union of finitely many hyperplanes and the connected
components are called \emph{Weyl chambers}. Let us fix a Weyl chamber $%
\mathfrak{a}^{+}$. With respect to this Weyl chamber a root $\alpha$ is said
to be positive if $\alpha (H) > 0$ for all $H \in \mathfrak{a}^+$. We denote
by $\Sigma^+$ the set of positive roots and by $\Sigma^+_0$ the set of
indivisible positive roots. If $\overline{\mathfrak{a}^{+}}$ is the closure
of $\mathfrak{a}^+$ then we denote by $\mathfrak{a}^{*}_{+}$ and $\overline{%
\mathfrak{a}^{*}_{+}}$ the cones corresponding to $\mathfrak{a}^+$ and $%
\overline{\mathfrak{a}^{+}}$ in $\mathfrak{a}^{\ast} $ (see \cite{WE} for
more details).

We have the Cartan decomposition of $G$: 
\begin{equation*}
G = K \exp \overline{\mathfrak{a}^+} K .
\end{equation*}

Define $m_{\alpha} : = \dim \mathfrak{g}_{\alpha} $ to be the multiplicity
of a root $\alpha$ and let 
\begin{equation}
\rho: = \frac{1}{2} \underset{\alpha \in \sum^+}{\sum} m_{\alpha} \alpha
\end{equation}
be half the sum of the positive roots counted according to their
multiplicity.

Let $x,y\in X$ and consider a base point $x_{0}\in X.$ Then, there are $%
g,h\in G$ such that $x=g(x_{0})$ and $y=h(x_{0}).$ Because of the Cartan
decomposition, there are $k_{1},k_{2}\in K,$ $H\in \overline{{a}^{+}}$ such that $%
g^{-1}h=k_{1}\exp {H}k_{2}.$ Then, the distance $d(x,y)$ of $x,y\in X$ is
given by 
\begin{equation}
d(x,y)=d(x_{0},\exp Hx_{0})=\left\Vert H\right\Vert .
\end{equation}

Recall that by the Cartan decomposition, the Haar measure on $G$ is written
	as 
	\begin{equation}
	\int_{G}f\left( g\right) dg=c\int_{K}dk_{1}\int_{\mathfrak{a}^{+}}\delta
	\left( H\right) dH\int_{K}f\left( k_{1}\left( \exp H\right) k_{2}\right)
	dk_{2},  \label{mes}
	\end{equation}%
	where, $dk$ is the normalised Haar measure of $K$ and the modular function $%
	\delta \left( H\right) $ satisfies: 
	\begin{equation*}
	\delta \left( H\right) \leq e^{2\left\langle \rho ,H\right\rangle },\text{ \ 
	}H\in \overline{\mathfrak{a}^{+}}.
	\end{equation*}
	From (\ref{mes}) it follows that if $f$ is $K$-bi-invariant, then 
	\begin{equation}
	\int_{G}f\left( g\right) dg=c\int_{{\overline{\mathfrak{a}^{+}}}}f\left( \exp
	H\right) \delta (H)dH.
	\end{equation}
	Recall that there are positive constants $c$ and $\alpha $ such that 
	\begin{equation}
	0<\phi _{-i\lambda }(\exp H)\leq c(1+\left\Vert H\right\Vert )^{a}e^{
		\lambda (H)-\rho (H)},\; H \in \overline{%
		\mathfrak{a}^{+}},\;\lambda \in \overline{\mathfrak{a}^{\ast}_{+}},  \label{spherical estimate}
	\end{equation}%
	see \cite{LOMAjga} for more details.

If $\Gamma \subset G$ is a discrete, torsion free subgroup of isometries of $%
X $, then the quotient space $M=\Gamma \backslash X$ equipped with the
projection of the Riemannian metric of $X$ is a Riemannian manifold and is
called \emph{locally symmetric space}. If $\pi :X\rightarrow M$ is the
canonical projection, we write $\tilde{x}=\pi (x)$. The distance $d_{M}$ on $%
M$ is given by 
\begin{equation}
d_{M}(\tilde{x},\tilde{y})=\min_{\gamma \in \Gamma }d(x,\gamma y).
\label{distance function on M}
\end{equation}

\subsection{The heat kernel on symmetric and locally symmetric spaces}\label{The heat kernel}



Denote by $h_{t}$ the heat kernel on the symmetric space $X$. The heat kernel on symmetric spaces has been extensively studied \cite%
{ANJ,ANO}. Sharp estimates of the heat kernel have been real hyperbolic
space have been obtained in \cite{DAVMAN} while in \cite{ANJ} Anker and Ji
and Anker and Ostellari in \cite{ANO}, generalized results of \cite{DAVMAN}
to all symmetric spaces of noncompact type. They proved the following sharp
estimate 
\begin{equation}
\begin{split}
h_{t}(\exp {H})& \asymp t^{-n/2}\left( \underset{\alpha \in \Sigma _{0}^{+}}{%
\prod }(1+\langle \alpha ,H\rangle )(1+t+\langle \alpha ,H\rangle
)^{\frac{m_{\alpha }+m_{2\alpha }}{2}-1}\right) \\
& \times e^{-\left\Vert \rho \right\Vert ^{2}t-\langle \rho ,H\rangle
-\left\Vert H\right\Vert ^{2}/(4t)},
\end{split}
\label{ostellari}
\end{equation}%
for all $H\in \overline{\mathfrak{a}^{+}}$ and all $t>0.$ Recall that we
write $f\asymp h$ for functions $f$ and $h$ if there is a positive constant $%
c>0$ such that $\frac{1}{c}h\leq f\leq ch.$

Set 
\begin{equation}
m=\underset{a\in \Sigma {}_{0}{}^{+}}{\sum }\left( \frac{m{}_{\alpha
}+m{}_{2\alpha }}{2}-1\right) \text{ and }A=\underset{a\in \Sigma_{0}^{+}}{\sum }\frac{m_{\alpha }+m_{2\alpha }}{2}.  \label{m index}
\end{equation}%
Note that (\ref{ostellari}) implies the following estimate%
\begin{equation}
h_{t}(\exp {H})\leq ct^{-n/2}(1+t)^{m}(1+\left\Vert H\right\Vert
)^{A}e^{-\left( \left\Vert \rho \right\Vert ^{2}t+\langle \rho ,H\rangle
+\left\Vert H\right\Vert {}^{2}/(4t)\right) },  \label{ankerest1}
\end{equation}%
for all $H\in \overline{\mathfrak{a}^{+}}$ and all $t>0.$

In \cite{DAVMAN}, Davies and Mandouvalos obtained heat kernel estimates on quotients of
the hyperbolic spaces, and in \cite{WE} Weber generalized these results on
locally symmetric spaces.

Estimates of the time derivatives of the heat kernel are obtained by Mandouvalos
and Tselepidis in \cite{MANTSE} for the case of real hyperbolic spaces. In \cite{GR2}, Grigory'an derived Gaussian upper bounds for all time
derivatives of the heat kernel, under some assumptions on the on-diagonal
upper bound for $h_{t}$ on an arbitrary complete non-compact Riemannian
manifold $X$. More precisely, it is proved that if there exists an increasing
continuous function $f(t)>0,$ $t>0,$ such that 
\begin{equation*}
h_{t}(x,x)\leq \frac{1}{f(t)},\text{ for all }t>0\text{ and }x\in X,
\end{equation*}%
then, 
\begin{equation}
\left\vert \frac{\partial ^{i}h_{t}}{\partial t^{i}}\right\vert (x,y)\leq 
\frac{1}{\sqrt{f(t)f_{2i}(t)}},\text{ for all }i\in \mathbb{N},\text{ }t>0,%
\text{ }x,y\in X,  \label{Grigory'an}
\end{equation}%
where the sequence of functions $f_{i}=f_{i}(t),$ is defined by 
\begin{equation}\label{sequencef}
f_{0}(t)=f(t)\text{ and }f_{i}(t)=\int_{0}^{t}f_{i-1}(s)ds,\;i\geq 1.
\end{equation}

We shall now apply the estimate (\ref{Grigory'an}) of Grigory'an in the case
of symmetric spaces of noncompact type.

We note that according to (\ref{ankerest}), 
	\begin{equation*}
	h_{t}(\exp H)\leq t^{-n/2}(1+t)^{m}.
	\end{equation*}
	Let $\Sigma^+$ be the set of positive roots and $l$ be the rank of $X$. Then, it holds $n=l+\sum\limits_{\alpha\in \Sigma^+}m_{\alpha}$, \cite{ANO}. It follows from (\ref{m index}) that  $n/2>m$. Then, 	
	\begin{equation*}
	f(t)={t^{n/2}}{(1+t)^{-m}}.
	\end{equation*}is an increasing function, thus we can invoke (\ref{Grigory'an}).
	
	By an induction argument we get that 
	\begin{equation}  \label{fi}
	f_{i}(t)\geq t^{(n/2)+i}{(1+t)^{-m}}.
	\end{equation}
	
	Then, we get the following result.
\begin{lemma}
	\label{Step 0 estimates} Suppose that $X$ is a symmetric space of noncompact
	type. For all $i\in \mathbb{N}$ there is a constant $c>0$ such that 
	\begin{equation}
	\left\vert \frac{\partial ^{i}h_{t}}{\partial t^{i}}(\exp {H})\right\vert
	\leq ct^{-(n/2)-i}(1+t)^{m},\text{ for all }t>0,\text{ }H\in \overline{%
		\mathfrak{a}^{+}},
	\end{equation}%
	where $m$ is defined by (\ref{m index}).
\end{lemma}

	\section{Estimates of the time derivatives of the heat kernel on symmetric spaces}\label{sectionsym}
Let $X$ be a noncompact symmetric space. Recall the heat kernel estimate
\begin{equation}
h_{t}(\exp {H})\leq ct^{-n/2}(1+t)^{m}(1+\left\Vert H\right\Vert
)^{A}e^{-\left( \left\Vert \rho \right\Vert ^{2}t+\langle \rho ,H\rangle
	+\left\Vert H\right\Vert {}^{2}/(4t)\right) },  \label{ankerest}
\end{equation}%
for all  $t>0$ and $H\in \overline{\mathfrak{a}^{+}}$, where $m$ and $A$ are defined in (\ref{m index}).

In this section we shall prove the main result, stated in Theorem \ref{Main result}. For the proof of (\ref{required estimate}) we need several lemmata. The
following lemma is technical but important for the proof of Theorem \ref%
{Main result}. Roughly speaking, according to the following result, an estimate for a function and its second derivative implies an estimate for its first derivative. 

\begin{lemma}
	\label{Porper} Let 
	\begin{equation}
	\alpha >\beta \geq 0,\text{ }D\geq D_{\ast },\text{ }B\geq B_{\ast },\text{ }%
	C\geq C_{\ast },  \label{conditions Porper}
	\end{equation}%
	and assume that for fixed $H\in \overline{\mathfrak{a}^{+}}$ the function $%
	f_{H}:(0,+\infty )\rightarrow \mathbb{R}$, satisfies 
	\begin{equation}
	\left\vert f_{H}(t)\right\vert \leq ct^{-\alpha }(1+t)^{\beta }(1+\left\Vert
	H\right\Vert )^{\gamma }e^{-Dt-B\langle \rho ,H\rangle -C{\left\Vert
			H\right\Vert ^{2}}/{(4t)}}  \label{fH1}
	\end{equation}%
	and 
	\begin{equation}
	\left\vert \frac{d^{2}f_{H}}{dt^{2}}(t)\right\vert \leq ct^{-\alpha
		-2}(1+t)^{\beta }(1+\left\Vert H\right\Vert )^{\gamma }e^{-D_{\ast
		}t-B_{\ast }\langle \rho ,H\rangle -C_{\ast }{\left\Vert H\right\Vert ^{2}}/{%
			(4t)}}.  \label{fH2}
	\end{equation}%
	Then, for all $\epsilon\in(0,1)$, there is a constant $c=c(\epsilon)>0$, such that for all $%
	H\in \overline{\mathfrak{a}^{+}},$ 
	\begin{equation*}
	\begin{split}
	\left\vert \frac{df_{H}}{dt}(t)\right\vert &\leq ct^{-\alpha -1}(1+t)^{\beta
	}(1+\left\Vert H\right\Vert )^{\gamma }\\
	&\times e^{-\left( {(D_{\ast }+D)t}/{2}+{%
			(B_{\ast }+B)\langle \rho ,H\rangle }/{2}+{(C_{\ast }+C\lambda _{\epsilon })}%
		{\left\Vert H\right\Vert ^{2}}/{8t}\right) },
	\end{split}
	\end{equation*}%
	where $\lambda _{\epsilon }=\frac{1-\epsilon }{1+\epsilon }.$
\end{lemma}

\begin{proof}
	Firstly, for all $\delta >0$, the mean value theorem yields for some $\theta \in (t, t+\delta)$:
	\begin{equation*}  
	\left|\frac{df_{H}}{dt}(\theta)\right|= \frac{1}{\delta} |f_H(t+\delta)-f_H(t)|\leq  \frac{1}{\delta}\left(
	\left|f_{H}(t)\right| + \left| f_{H}(t+\delta)\right|\right).
	\end{equation*}
	Applying once again the mean value theorem, now on the function $\frac{df_H}{dt}$ on $[t, \theta]$	we have
	\[
	\left| \frac{df_H}{dt}(t) \right|\leq \left| \frac{df_H}{dt}(\theta) \right|+\delta\sup_{\tau \in (t,t+\delta)}\left|\frac{d^{2}f_{H}}{dt^2}(\tau)\right| .\]
	It follows that
	\begin{equation}  \label{mvt}
	\left|\frac{df_{H}}{dt}(t)\right| \leq \frac{1}{\delta}\left(
	\left|f_{H}(t)\right| + \left| f_{H}(t+\delta)\right|\right) + \delta
	\sup_{\tau \in (t,t+\delta)}\left|\frac{d^{2}f_{H}}{dt^2}(\tau)\right| .
	\end{equation}
	Note that $t\longrightarrow t^{-\alpha }(1+t)^{\beta }$ is a decreasing
	function of $t$, since $\alpha >\beta ,$ therefore 
	\begin{equation*}
	(t+\delta )^{-\alpha }(1+t+\delta )^{\beta }\leq t^{-\alpha }(1+t)^{\beta }.
	\end{equation*}%
	Note also that {$\langle \rho,H \rangle \geq \rho_{m}\|H\|>0$, for $H\neq 0$.} Thus (\ref{mvt}) and the estimates (\ref{fH1}) and (\ref{fH2}) imply that 
	\begin{equation*}
	\begin{split}
	\left\vert \frac{df_{H}}{dt}(t)\right\vert & \leq c\frac{1}{\delta }%
	t^{-\alpha }(1+t)^{\beta }(1+\|H\|)^{\gamma }e^{-Dt-B\langle \rho ,H\rangle -C{%
			\left\Vert H\right\Vert ^{2}}/{4(t+\delta )}}+ \\
	& +c\delta t^{-\alpha -2}(1+t)^{\beta }(1+\|H\|)^{\gamma }e^{\left( -D_{\ast
		}t-B_{\ast }\langle \rho ,H\rangle -C_{\ast }{\left\Vert H\right\Vert ^{2}}/{%
			4(t+\delta )}\right) }.
	\end{split}%
	\end{equation*}
	
	Choose now 
	\begin{equation*}
	\delta= \epsilon t e^{ -{(D-D_{*})t}/{2} - {(B-B_{*})\langle \rho, H
			\rangle }/{2} -(C-C_{*}){\left\Vert H\right\Vert^2}/{(8 t)}  }.
	\end{equation*}
	
	Thus, 
	\begin{align}
	\left\vert \frac{df_{H}}{dt}(t)\right\vert  \leq & c\frac{1}{\epsilon }%
	t^{-\alpha -1}(1+t)^{\beta }(1+\|H\|)^{\gamma } \notag \\
	&\times e^{-(D+D_{\ast })t/{2}-(B+B_{\ast
		}) \langle \rho ,H\rangle /{2}+(C-C_{\ast }) \left\Vert H\right\Vert
		^{2}/(8t)-C\left\Vert H\right\Vert ^{2}/4(t+\delta )}+ \notag \\
	& +c\epsilon t^{-\alpha -1}(1+t)^{\beta }(1+\|H\|)^{\gamma } \notag \\
	&\times e^{ -(D+D_{\ast }) t/%
		{2}-(B+B_{\ast }) \langle \rho ,H\rangle/{2} -(C-C_{\ast }) \left\Vert
		H\right\Vert ^{2}/{(8t)}-C_{\ast }\left\Vert H\right\Vert ^{2}/4(t+\delta )}\label{inequality0}
	\end{align}
	From (\ref{conditions Porper}), it follows that $\delta \leq \epsilon t.$
	Thus 
	\begin{equation*}
	\frac{1}{2t}-\frac{1}{t+\delta }\leq -\frac{1-\epsilon }{2t(1+\epsilon )}=-%
	\frac{\lambda _{\epsilon }}{2t}.
	\end{equation*}%
	Consequently, 
	\begin{equation}
	\frac{C-C_{\ast }}{2}\frac{\left\Vert H\right\Vert ^{2}}{4t}-C\frac{%
		\left\Vert H\right\Vert ^{2}}{4(t+\delta )}\leq -\frac{\left\Vert
		H\right\Vert ^{2}}{4t}\frac{C\lambda _{\epsilon }+C_{\ast }}{2},
	\label{inequality1}
	\end{equation}%
	and similarly 
	\begin{eqnarray}
	\frac{C-C_{\ast }}{2}\frac{\left\Vert H\right\Vert ^{2}}{4t}+\frac{C_{\ast
		}\left\Vert H\right\Vert ^{2}}{4(t+\delta )} &\geq &\frac{\left\Vert
		H\right\Vert ^{2}}{4t}\frac{C_{\ast }\lambda _{\epsilon }+C}{2}  \notag \\
	&\geq &\frac{\left\Vert H\right\Vert ^{2}}{4t}\frac{C\lambda _{\epsilon
		}+C_{\ast }}{2}.  \label{inequality2}
	\end{eqnarray}%
	Thus, from (\ref{inequality0}), (\ref{inequality1}) and (\ref{inequality2})
	it follows that 
	\begin{equation*}
	\begin{split}
	\left\vert \frac{df_{H}}{dt}(t)\right\vert &\leq  c(\epsilon)t^{-\alpha -1}(1+t)^{\beta
	}(1+\left\Vert H\right\Vert )^{\gamma }\\
	& \times e^{-\left( {(D_{\ast }+D)t}/{2}+{%
			(B_{\ast }+B)\langle \rho ,H\rangle }/{2}+{(C_{\ast }+C\lambda _{\epsilon })}%
		{\left\Vert H\right\Vert ^{2}}/{8t}\right) },
	\end{split}
	\end{equation*}
	where $c(\epsilon)=c(1/\epsilon+\epsilon)$, with $c>0$  constant.
\end{proof}

Next, we use an inductive argument. More precisely, we apply Lemma \ref{Porper} and we are improving step by step the resulting estimates by using an iterative argument. 

\begin{lemma}
	\label{Technical Lemma} Suppose that $X$ is a symmetric space of noncompact
	type. Let us fix {$\epsilon \in (0,1)$} and set $\lambda _{\epsilon }=\frac{1-\epsilon }{1+\epsilon }$. Then, for all $i,\ell \in \mathbb{N},$ there are
	constants $c,\beta _{\ell}^{i}$ and $\gamma _{\ell}^{i}>0$ such that 
	\begin{align}
	\left\vert \frac{\partial ^{i}h_{t}}{\partial t^{i}}(\exp H)\right\vert &\leq 
	ct^{-(n/2)-i}(1+t)^{m}(1+\left\Vert H\right\Vert )^{A}\notag\\
	&\times e^{-\beta
		_{\ell}^{i}\left( \left\Vert \rho \right\Vert ^{2}t+\langle \rho ,H\rangle
		\right) }e^{-\gamma _{\ell}^{i}{\left\Vert H\right\Vert ^{2}}/{(4t)}},
	\label{sequences estimate}
	\end{align}
	for all $t>0$  and $H\in \overline{\mathfrak{a}_{+}}$, where $%
	m,A $ are defined in (\ref{m index}) and $c$ is a constant that depends on $%
	\epsilon ,i,\ell$. Furthermore, the sequences $\beta _{\ell}^{i},\gamma _{\ell}^{i}$
	satisfy the iteration formulas 
	\begin{equation}
	\begin{split}
	\beta _{\ell}^{i}& =\frac{1}{2}(\beta _{\ell-1}^{i-1}+\beta _{\ell-1}^{i+1}),
	\\
	\gamma _{\ell}^{i}& =\frac{1}{2}(\lambda _{\epsilon}\gamma
	_{\ell-1}^{i-1}+\gamma _{\ell-1}^{i+1}),
	\end{split}
	\label{recurrence relations}
	\end{equation}%
	the inequalities $\beta_{\ell }^{i-1}\geq \beta_{\ell }^{i+1}$, $\gamma_{\ell }^{i-1}\geq \gamma_{\ell }^{i+1}$
	and the initial conditions 
	\begin{equation}
	\beta _{0}^{i}=0,\gamma _{0}^{i}=0,\text{ for all }i\geq 1,\;\beta
	_{\ell}^{0}=1,\gamma _{\ell}^{0}=1,\text{ for all }\ell\geq 0.
	\label{initial conditions}
	\end{equation}
\end{lemma}

\begin{proof}
	For every $\ell\in \mathbb{N}$, consider the following statement $L(\ell)$:
	for all $i\in \mathbb{N},$ $\frac{\partial ^{i}h_{t}}{\partial t^{i}}(\exp
	H) $ satisfies the estimate (\ref{sequences estimate}) and the constants $%
	\beta _{\ell}^{i},\gamma _{\ell}^{i}$ appearing in (\ref{sequences estimate}%
	) satisfy the iteration formulas (\ref{recurrence relations}) with
	initial conditions (\ref{initial conditions}), and $\beta_{\ell }^{i-1}\geq \beta_{\ell }^{i+1}$, $\gamma_{\ell }^{i-1}\geq \gamma_{\ell }^{i+1}$. We shall then prove by
	induction, that $L(\ell)$ holds for every $\ell\in \mathbb{N}$.
	
	For $\ell=0$ we have to prove that for all $i\in \mathbb{N},$ $\frac{%
		\partial ^{i}h_{t}}{\partial t^{i}}(\exp H)$ satisfies the estimate (\ref%
	{sequences estimate}) and that the constants $\beta _{0}^{i},\gamma
	_{0}^{i}$ satisfy $\beta _{0}^{i}=\gamma _{0}^{i}=0,\text{ for all }i\geq
	1,\text{ and }\beta _{0}^{0}=\gamma _{0}^{0}=1.$ {Also, we need to show that $\beta_{0 }^{i-1}\geq \beta_{0 }^{i+1}$, $\gamma_{0 }^{i-1}\geq \gamma_{0 }^{i+1}$}.
	
	Indeed, from Lemma \ref{Step 0 estimates} we get that for $i\geq 1$ 
	\begin{equation}
	\left\vert \frac{\partial ^{i}h_{t}}{\partial t^{i}}(\exp {H})\right\vert
	\leq ct^{-(n/2)-i}(1+t)^{m},\text{ for all }t>0,\text{ }H\in \overline{%
		\mathfrak{a}^{+}}.
	\end{equation}%
	But, $1\leq (1+\left\Vert H\right\Vert )^{A}.$ Thus 
	\begin{equation*}
	\left\vert \frac{\partial ^{i}h_{t}}{\partial t^{i}}(\exp {H})\right\vert
	\leq ct^{-(n/2)-i}(1+t)^{m}(1+\left\Vert H\right\Vert )^{A},\text{ for all }%
	t>0,\text{ }H\in \overline{\mathfrak{a}^{+}},
	\end{equation*}
	i.e. (\ref{sequences estimate}) holds true for all $i\geq 1$, with $\beta
	_{0}^{i}=\gamma _{0}^{i}=0.$ Furthermore, from the estimates of the heat
	kernel in (\ref{ankerest}), we obtain that 
	\begin{equation*}
	\left\vert h_{t}(\exp {H})\right\vert \leq ct^{-n/2}(1+t)^{m}(1+\left\Vert
	H\right\Vert )^{A}e^{-\left( \left\Vert \rho \right\Vert ^{2}t+\langle \rho
		,H\rangle +\left\Vert H\right\Vert ^{2}/4t\right) }.
	\end{equation*}%
	Thus (\ref{sequences estimate}) holds true also for $i=0$ and $\beta_{0}^{0}=\gamma _{0}^{0}=1.$ {Last, from Lemma \ref{Step 0 estimates}, for $i\geq 1$, we have $\beta_{0 }^{i-1}\geq 0 =\beta_{0 }^{i+1}$, $\gamma_{0 }^{i-1}\geq 0= \gamma_{0 }^{i+1}$.} Therefore the statement $L(0)$ holds true.
	
	Let us assume now that $L(\ell-1)$ holds true. Thus, for all $i\in \mathbb{N}
	$, there are constants $c,\;\beta _{\ell-1}^{i},\;\gamma _{\ell-1}^{i}>0$ such
	that $\frac{\partial ^{i}h_{t}}{\partial t^{i}}(\exp H)$ satisfies the
	estimate (\ref{sequences estimate}). In addition, {$\beta_{\ell -1 }^{i-1}\geq \beta_{\ell -1 }^{i+1}$, $\gamma_{\ell -1 }^{i-1}\geq \gamma_{\ell -1 }^{i+1}.$}
	
	We shall prove that $L(\ell)$ holds true. Indeed, from the estimates of the
	heat kernel in (\ref{ankerest}), we have that 
	\begin{equation*}
	|h_{t}(\exp {H})|\leq ct^{-n/2}(1+t)^{m}(1+\left\Vert H\right\Vert
	)^{A}e^{-\left( \left\Vert \rho \right\Vert ^{2}t+\langle \rho ,H\rangle
		+\left\Vert H\right\Vert^{2}/4t\right) }.
	\end{equation*}%
	Thus (\ref{sequences estimate}) holds true for $i=0$ with $\beta
	_{\ell}^{0}=\gamma _{\ell}^{0}=1.$
	
	For $i\geq 1$, consider the function 
	\begin{equation*}
	f_{H}(t)=\frac{\partial ^{i-1}h_{t}}{\partial t^{i-1}}(\exp {H}).
	\end{equation*}%
	From the validity of $L\left( \ell-1\right) $, we get that for $i-1$ and $%
	i+1 $ we have that 
	\begin{equation*}
	\begin{split}
	\left\vert f_{H}(t)\right\vert =  \left\vert \frac{\partial ^{i-1}h_{t}}{%
		\partial t^{i-1}}(\exp {H})\right\vert &\leq  ct^{-\alpha }(1+t)^{\beta
	}(1+\|H\|)^{\gamma }\\
	&\times e^{-Dt-B\langle \rho ,H\rangle +C{\left\Vert H\right\Vert
			^{2}}/{(4t)}},  \\
	\left\vert \frac{d^{2}f_{H}}{dt^{2}}(t)\right\vert =  \left\vert \frac{%
		\partial ^{i+1}h_{t}}{\partial t^{i+1}}(\exp {H})\right\vert &\leq 
	ct^{-\alpha -2}(1+t)^{\beta }(1+\|H\|)^{\gamma }\\
	&\times e^{-D_{\ast }t-B_{\ast }\langle
		\rho ,H\rangle +C_{\ast }{\left\Vert H\right\Vert ^{2}}/{(4t)}},
	\end{split}
	\end{equation*}
	with $\alpha =(n/2)+i-1,$ $\beta =m,$ $\gamma =A,$ $D=B=\beta
	_{\ell-1}^{i-1},$ $C=\gamma _{\ell-1}^{i-1}$ and $D_{\ast }=B_{\ast }=\beta
	_{\ell-1}^{i+1},$ $C_{\ast }=\gamma _{\ell-1}^{i+1}$. {Note that $B\geq B_{\ast}$ and $C\geq C_{\ast}$, from $L(\ell- 1)$.}
	
	Thus, by Lemma \ref{Porper}, applied for the function $f_{H}(t)$, it follows
	that 
	\begin{equation*}
	\begin{split}
	\left\vert \frac{df_{H}}{dt}(t)\right\vert =\left\vert \frac{\partial
		^{i}h_{t}}{\partial t^{i}}(\exp {H})\right\vert &\leq 
	ct^{-(n/2)-i}(1+t)^{m}(1+\left\Vert H\right\Vert )^{A}\\
	&\times e^{-\beta _{\ell
		}^{i}\left( \left\Vert \rho \right\Vert ^{2}t+\langle \rho ,H\rangle \right)
	}e^{-\gamma _{\ell }^{i}{\left\Vert H\right\Vert ^{2}}/{(4t)}},
	\end{split}
	\end{equation*}%
	for all $i\geq 1,$ where $\beta _{\ell }^{i}$ and $\gamma _{\ell }^{i}$
	satisfy (\ref{recurrence relations}). {Finally, from $L(\ell -1)$ and (\ref{recurrence relations})  it is straightforward that $\beta _{\ell}^{i-1}\geq \beta_{\ell }^{i+1}$ and $\gamma _{\ell}^{i-1}\geq \gamma _{\ell}^{i+1}$ }. Thus, the statement $L(\ell )$ is
	valid and the proof of the lemma is complete.
\end{proof}

\begin{remark}
	The constant $c=c(i,\ell,\epsilon)$ in relation (\ref{sequences estimate})
	of Lemma \ref{Technical Lemma} depends on $i, \ell$ and $\epsilon$ and it
	increases to infinity (when either $i\rightarrow \infty$ or $\ell\rightarrow
	\infty$ or $\epsilon\rightarrow 0$), but we only need the fact that it is
	finite for fixed $i,\ell, \epsilon$.
\end{remark}

Finally, we shall show that the estimates obtained are  precise, by proving that if $\ell\rightarrow \infty$ and $\epsilon \rightarrow 0$ then the exponents $\gamma_{\ell}^i $ and $\beta_{\ell}^i $ converge to 1. More precisely, we shall prove the following result. 

\begin{lemma}
	For any $i\in \mathbb{N},$ 
	\begin{equation}  \label{sequence convergence}
	\lim_{\ell\rightarrow \infty}\gamma_{\ell}^i = \left( 1 - \sqrt{1 -\lambda
		{}_{\epsilon}} \right)^i \text{ and } \lim_{\ell\rightarrow
		\infty}\beta_{\ell}^i = 1.
	\end{equation}
\end{lemma}

\begin{proof}
	We shall deal only with $\gamma_{\ell}^i$. The proof that $%
	\lim_{\ell\rightarrow \infty}\beta_{\ell}^i = 1$ is similar, and we shall
	omit it.
	
	\textit{Claim 1.} For every $\ell \in \mathbb{N}$ consider the following
	statement $L(\ell )$: for all $i\in \mathbb{N},$ 
	\begin{equation}
	\gamma _{\ell }^{i}\leq 1.
	\end{equation}%
	We shall prove by induction that $L(\ell )$ is valid for all $\ell \in 
	\mathbb{N}$.
	
	For $\ell=0$ we have to prove that for all $i\in \mathbb{N},$ we have that $%
	\gamma _{0}^{i}\leq 1.$ Indeed, this is a consequence of the initial
	conditions $\gamma _{0}^{i}=0$ and $\gamma _{0}^{0}=1.$ Thus $L(0)$ holds
	true.
	
	Let us assume now that $L(\ell-1)$ holds true. Thus, for all $i \in \mathbb{N%
	},$ we have that $\gamma_{\ell-1}^i \leq 1.$
	
	We shall prove that $L(\ell)$ holds true. Recall that by the induction
	assumption, for all $i\in \mathbb{N}$, for $i-1$ and $i+1$ we have that $%
	\gamma _{\ell-1}^{i-1}\leq 1$ and $\gamma _{\ell-1}^{i+1}\leq 1$. Thus, from
	(\ref{recurrence relations}) it follows that 
	\begin{equation*}
	\gamma _{\ell}^{i}=\frac{\lambda _{\epsilon }}{2}\gamma _{\ell-1}^{i-1}+%
	\frac{1}{2}\gamma _{\ell-1}^{i+1}\leq \frac{\lambda _{\epsilon }}{2}+\frac{1%
	}{2}\leq 1,
	\end{equation*}%
	thus the statement $L(\ell)$ is valid and this completes the proof of Claim
	1.
	
	\textit{Claim 2.} For every $\ell \in \mathbb{N}$ consider the following
	statement $L(\ell )$: for all $i\in \mathbb{N},$ 
	\begin{equation}
	\gamma _{\ell }^{i}\leq \gamma _{\ell +1}^{i}.
	\end{equation}%
	We shall prove that $L(\ell )$ is valid for all $\ell \in \mathbb{N}$. We
	proceed once again by induction in $\ell \in \mathbb{N}.$
	
	For $\ell=0$ we have to prove that for all $i\in \mathbb{N},$ $\gamma
	_{0}^{i}=0\leq \gamma _{1}^{i}.$ Indeed, from (\ref{initial conditions}) it
	follows that $\gamma _{0}^{i}=0\leq \gamma _{1}^{i},$ for all $i>0$ and $%
	\gamma _{0}^{0}=1=\gamma _{1}^{0}.$ Therefore the statement $L(0)$ holds
	true.
	
	Let us assume now that $L(\ell-1)$ holds true, i.e. that for all $i\in 
	\mathbb{N},$ $\gamma _{\ell-1}^{i}\leq \gamma _{\ell}^{i}.$
	
	We shall prove that $L(\ell)$ holds true, i.e. that for all $i\in \mathbb{N}%
	, $ $\gamma _{\ell}^{i}\leq \gamma _{\ell+1}^{i}.$ Recall that by (\ref%
	{recurrence relations}) we have that 
	\begin{equation}
	\gamma _{\ell}^{i}=\frac{1}{2}(\lambda _{\epsilon }\gamma
	_{\ell-1}^{i-1}+\gamma _{\ell-1}^{i+1}).  \label{itfor}
	\end{equation}%
	Then by the induction assumption for $i-1$ and $i+1$ we have that $\gamma
	_{\ell-1}^{i-1}\leq \gamma _{\ell}^{i-1}$ and $\gamma _{\ell-1}^{i+1}\leq
	\gamma _{\ell}^{i+1}.$ Hence, from (\ref{itfor}) we get that 
	\begin{equation*}
	\gamma _{\ell}^{i}\leq \frac{1}{2}(\lambda _{\epsilon }{\gamma _{\ell}}%
	^{i-1}+\gamma _{\ell}^{i+1})=\gamma _{\ell+1}^{i}.
	\end{equation*}%
	Thus the statement $L(\ell)$ is valid and the proof of Claim 2 is complete.
	
	\textit{Claim 3.} For all $i\in \mathbb{N},$ 
	\begin{equation}  \label{claim3}
	\lim_{\ell\rightarrow \infty}\gamma_{\ell}^i = \left( 1 - \sqrt{1 -\lambda
		{}_{\epsilon}} \right)^i .
	\end{equation}
	
	Note that by Claim 2, the sequence $\gamma _{\ell }^{i}$ is increasing in $%
	\ell $ and by Claim 1, $\gamma _{\ell }^{i}$ is bounded above. Thus $%
	\lim_{\ell \rightarrow \infty }\gamma _{\ell }^{i}$ exists and since $0\leq
	\gamma _{\ell }^{i}\leq 1,$ then 
	\begin{equation*}
	\lim_{\ell \rightarrow \infty }\gamma _{\ell }^{i}=\gamma _{i}\leq 1.
	\end{equation*}%
	Note that $\gamma _{\ell }^{0}=1$, for all $\ell \in \mathbb{N},$ thus $%
	\gamma _{0}=1.$
	
	Now, taking limits in the iteration formula (\ref{itfor}) we obtain that 
	\begin{equation*} 
	\gamma _{i}=\frac{1}{2}(\lambda _{\epsilon}\gamma _{i-1}+\gamma _{i+1}),
	\end{equation*}%
	thus 
	\begin{equation*}
	\gamma _{i+1}-2\gamma _{i}+\lambda _{\epsilon }\gamma _{i-1}=0.
	\end{equation*}%
	This is a homogeneous linear recurrence relation with constant coefficients
	and the solutions of this equation are given by 
	\begin{equation*}
	\gamma _{i}=C_{1}r_{1}^{i}+C_{2}r_{2}^{i},\text{ \ }C_{1},C_{2}\in \mathbb{R}%
	,
	\end{equation*}%
	where $r_{1},r_{2}$ are the roots of the equation 
	\begin{equation*}
	r^{2}-2r+\lambda _{\epsilon }=0.
	\end{equation*}%
	Thus, we conclude that 
	\begin{equation}
	\gamma _{i}=C_{1}\left( 1-\sqrt{1-\lambda {}_{\epsilon}}\right)
	^{i}+C_{2}\left( 1+\sqrt{1-\lambda {}_{\epsilon}}\right) ^{i},
	\label{gammai}
	\end{equation}%
	for some $C_{1},C_{2}\in \mathbb{R}.$
	
	Since $0\leq \gamma _{i}\leq 1$, we get $C_{2}=0,$ otherwise $%
	\lim_{i\rightarrow \infty }\gamma _{i}=\infty $. Also, since $\gamma _{0}=1$%
	, we get $C_{1}=1.$ Thus, from (\ref{gammai}) for $C_{1}=1,$ $C_{2}=0,$ we
	get (\ref{claim3}) and the proof is complete.
\end{proof}

\textit{End of the proof of Theorem \ref{Main result}}: To complete the
proof of Theorem \ref{Main result}, notice that $\lim_{\epsilon\rightarrow 0}\left( 1-\sqrt{1-\lambda {}_{\epsilon}}\right) ^{i}=1$%
. Thus, taking $\ell\in \mathbb{N}$ sufficiently large and $\epsilon$
sufficiently close to zero, one has $\gamma _{\ell}^{i}\geq 1-\epsilon $ and 
$\beta _{\ell}^{i}\geq 1-\epsilon .$ Thus, from (\ref{sequences estimate})
and (\ref{sequence convergence}) it follows that 
\begin{equation*}
\left\vert \frac{\partial ^{i}h_{t}}{\partial t^{i}}(x,y)\right\vert \leq
ct^{-(n/2)-i}(1+t)^{m}(1+\left\Vert H\right\Vert )^{A}e^{-(1-\epsilon
	)\left( \left\Vert \rho \right\Vert ^{2}t+\langle \rho ,H\rangle +\left\Vert
	H\right\Vert ^{2}/(4t)\right) }.
\end{equation*}%
Taking now into account that if $\alpha,\beta >0$, then there exists a
constant $c=c(\alpha ,\beta )$ such that $x^{\alpha }\leq ce^{\beta x}$ for
all $x>0,$ we conclude that for every $\epsilon >0,$ there exists a constant 
$c>0$ such that 
\begin{equation*}
\left\vert \frac{\partial ^{i}h_{t}}{\partial t^{i}}(x,y)\right\vert \leq
ct^{-(n/2)-i}e^{-(1-\epsilon )\left( \left\Vert \rho \right\Vert
	^{2}t+\langle \rho ,H\rangle +\left\Vert H\right\Vert ^{2}/(4t)\right) },
\end{equation*}%
and the proof of Theorem \ref{Main result} is complete.

\begin{remark}
	If $X$ is a Cartan-Hadamard manifold then, \cite{GR3}, the heat kernel $%
	h_{t} $ of $X$ satisfies  pointwise bounds of the type 
	\begin{equation*}
	h_{t}(x,y)\leq \frac{c}{\min\{1,t^{\alpha }\}}e^{-At-Bd(x,y)-C{d^{2}(x,y)}/{t}%
	},\text{ }t>0,\text{ }x,y\in X,
	\end{equation*}%
	for some positive constants $c,A,B,C$ and $\alpha $. Proceeding as in
	Section \ref{sectionsym}, one can prove the following
	estimate: for all $\epsilon\in(0,1)$ and $i\in \mathbb{N},$ there is a constant $%
	c>0$ such that 
	\begin{equation*}
	\left\vert \frac{\partial ^{i}h_{t}}{\partial t^{i}}(x,y)\right\vert \leq 
	\frac{c}{\min\{1,t^{\alpha +i}\}}e^{-(1-\epsilon )\left( At+Bd(x,y)+{%
			Cd^{2}(x,y)}/{t}\right) }.
	\end{equation*}
\end{remark}

\section{Applications}\label{applications}

\subsection{Estimates of the gradient of the heat kernel on symmetric spaces}

As a direct application of Theorem \ref{Main result}, we obtain the gradient
estimates of $h_{t}$ given in Corollary \ref{Gradient estimate}.

\begin{proof}[Proof of Corollary \ref{Gradient estimate}.] Let us recall that if $%
	X $ is a complete, non-compact, $n$-dimensional Riemannian manifold, with
	Ricci curvature bounded from below by $-R^{2}$, then by \cite{LIYAU} for $%
	\gamma >1$, we have that 
	\begin{equation}  \label{gradient equation}
\left\Vert \nabla h_t (x, y) \right\Vert^2 \leq h_t^2 (x, y) \left( \frac{n
	R^2 \gamma^2}{\sqrt{2} (\gamma - 1)} + \frac{n \gamma^2}{2 t} \right) +\gamma h_t
(x, y) \left| \frac{\partial h_t (x, y)}{\partial t}\right|.
\end{equation}
	\text{for all }$t>0,\;x,y\in X.$ 

Using the heat kernel estimate (\ref{ankerest}) as well as Theorem \ref%
	{Main result}, and inequality (\ref{gradient equation}), the result follows.
\end{proof}

\subsection{Estimates of the time derivatives of the heat kernel on locally symmetric spaces}\label{Locally Symmetric Space}

In this section we obtain estimates of the heat kernel time derivatives in
the case of a locally symmetric space $M = \Gamma \backslash X$, and prove Theorem \ref{Locally symmetric 2}.  Our results
extend the estimates of Weber (see \cite{WE}).

Recall that
\begin{equation*}
\rho _{m}=\min_{H\in \overline{\mathfrak{a}_{+}},\text{ }\left\Vert
	H\right\Vert =1}\langle \rho ,H\rangle .
\end{equation*}
Suppose that $\delta (\Gamma )<\left\Vert\rho \right\Vert+\rho _{m}$. 

The heat
kernel $h_{t}^{M}(\tilde{x},\tilde{y})$ on $M$ is given by the formula 
\begin{equation}
h_{t}^{M}(\tilde{x},\tilde{y})=\sum_{\gamma \in \Gamma }h_{t}(x,\gamma
y),\,\,\,\text{ for all }x,y\in X,\text{ \ }t>0,  \label{summation formula}
\end{equation}%
\cite{WE}.

Recall also
that
\begin{equation*}
d_M (\tilde{x}, \tilde{y}) = \inf_{\gamma\in \Gamma} d (x, \gamma y) .
\end{equation*}
Consider $\alpha _{2}\in (\delta (\Gamma ),\left\Vert\rho \right\Vert+\rho
_{m})$ and $\alpha _{1},\alpha _{3}\in [0,1]$ such that \[\alpha _{1}\alpha
_{3}\in \left[ \left( \frac{\alpha _{2}-\rho _{m}}{\left\Vert\rho \right\Vert%
}\right) ^{2},1\right].\]
We shall now prove Theorem \ref{Locally symmetric 2}.

\begin{proof}[Proof of Theorem \ref{Locally symmetric 2}.]
	According to Theorem \ref{Main result}, 
	\begin{equation*}
	\left|\frac{\partial ^{i}h_{t}}{\partial t^{i}}(\exp H)\right|\leq ct^{-{%
			(n/2)}-i}e^{-(1-\epsilon )\left( \left\Vert\rho \right\Vert^{2}t+\langle
		\rho ,H\rangle +{\left\Vert H\right\Vert{}^{2}}/{(4t)}\right) }.
	\end{equation*}%
	Thus, taking into account that $d(x,y)=\left\Vert H\right\Vert, $ as well as that $\langle \rho, H \rangle \geq \rho_{m}\|H\|$, we obtain
	the estimate 
	\begin{equation}
	\left|\frac{\partial ^{i}h_{t}}{\partial t^{i}}(x,y)\right| \leq ct^{-{(n/2)}%
		-i}e^{-(1-\epsilon)\left( \left\Vert\rho \right\Vert^{2}t+\rho
		_{m}d(x,y)+{d^2(x,y)}/{(4t)}\right) }.  \label{eqn locally symmetric}
	\end{equation}
	Note that 
	\begin{equation}  \label{eqn quadratic}
	\alpha _{1}\left\Vert\rho \right\Vert^{2}t+(\rho _{m}-\alpha _{2})d_{M}(%
	\tilde{x},\tilde{y})+\alpha _{3}\frac{d_{M}^{2}(\tilde{x},\tilde{y})}{4t}
	\geq 0,
	\end{equation}%
	since $\alpha _{1}\alpha _{3}\in \left[ \left( \frac{\alpha _{2}-\rho _{m}}{%
		\left\Vert\rho \right\Vert}\right) ^{2},1\right], $ with equality when $%
	\alpha_{2}=\rho_{m}>\delta(\Gamma)$ and $\alpha_{1}=\alpha_{3}=0$.
	
	Thus, from (\ref{eqn locally symmetric}) and (\ref{eqn quadratic}) it
	follows that 
	\begin{align}
	\left|\frac{\partial ^{i}h_{t}}{\partial t^{i}}(x,y)\right| &\leq  ct^{-{%
			(n/2)}-i}e^{-(1-\epsilon)\left( (1-\alpha _{1})\left\Vert\rho
		\right\Vert^{2}t+ \alpha_{2}d(x,y)+(1-\alpha _{3}){d^{2}(x,y)}/{(4t)}\right)} \notag
	\\
	&= ct^{-{(n/2)}-i}e^{-(1-\epsilon )\left( (1-\alpha _{1})\left\Vert\rho
		\right\Vert^{2}t+ (\alpha_{2}-{(\delta (\Gamma)+\epsilon)}/{(1-\epsilon)}%
		)d(x,y)+(1-\alpha _{3}){d^{2}(x,y)}/{(4t)}\right)} \notag \\
	& \times e^{-(\delta (\Gamma)+\epsilon)d(x,y)}.
	\end{align}
	A summation argument implies that for every $\epsilon\in (0,1)$ we have 
	\begin{equation*}
	\begin{split}
	\left| \frac{\partial ^{i}h_{t}^{M}}{\partial t^{i}}(\tilde{x},\tilde{y})
	\right| &\leq  ct^{-(n/2)-i}e^{-(1-\epsilon)\left( (1-\alpha_{1}
		)\left\Vert\rho \right\Vert^{2}t+(\alpha_{2}-\delta(\Gamma))d_{M}(\tilde{x},%
		\tilde{y})+(1-\alpha_{3} )d_{M}^{2}(\tilde{x},\tilde{y})/(4t) \right)} \\
	& \times P_{\epsilon +\delta (\Gamma )}(x,y),
	\end{split}%
	\end{equation*}
	and the proof of Theorem \ref{Locally symmetric 2} is complete.
\end{proof}
\begin{remark}
	Recall that the kernel $p_{t}^{M}$ of the Poisson semigroup $e^{-t\sqrt{%
			-\Delta }}$ on a Riemannian manifold $M$ is given by the following
	subordination formula: 
	\begin{equation}\label{subord}
	p_{t}^{M}(\tilde{x},\tilde{y})=\frac{t}{2\sqrt{\pi }}\int_{0}^{\infty
	}u^{-3/2}e^{-t^{2}/4u}h_{u}^{M}(\tilde{x},\tilde{y})du,
	\end{equation}%
	\cite[p.1075]{ANJ}, where $h_{u}^{M}(\tilde{x},\tilde{y})$ is the heat kernel of $M$. 
	Using the estimates of Theorem \ref{Locally symmetric 2}, we obtain the following pointwise estimates for the Poisson kernel 
	\begin{equation}\label{poisson}
	\begin{split}
	p_{t}^{M}(\tilde{x},\tilde{y})& \leq \frac{ct}{\left( d_{M}^{2}(\tilde{x},\tilde{y})(1-\alpha
		_{3})+\frac{t^{2}}{1-\epsilon}\right) ^{(n+1)/2}}P_{\epsilon +\delta (\Gamma )}(x,y) \\
	& \times e^{-\left( 1-\epsilon \right) \left( (\alpha _{2}-\delta (\Gamma
		))d_{M}(\tilde{x},\tilde{y})+\Vert \rho \Vert \sqrt{1-\alpha _{1}}\sqrt{d_{M}^{2}(\tilde{x},\tilde{y})(1-%
			\alpha _{3})+\frac{t^2}{1-\epsilon}}\right) },
	\end{split}%
	\end{equation}
	\text{for all }$\tilde{x},\tilde{y}\in M,$ $t>0.$ 
	Similarly, one can prove that the kernel 
	\begin{equation*}
	r_{s}^{M}\left( \tilde{x},\tilde{y}\right) =\Gamma (s)^{-1}\int_{0}^{\infty
	}t^{s-1}h_{t}^{M}\left( \tilde{x},\tilde{y}\right) dt\text{, }s>0,
	\end{equation*}
	satisfies
	\begin{equation*}
	r_{s}^{M}(\tilde{x},\tilde{y})\leq  cP_{\epsilon +\delta (\Gamma )}(x,y)e^{-\left( 1-\epsilon \right) \left( (\alpha
		_{2}-\delta (\Gamma ))+\Vert \rho \Vert \sqrt{(1-\alpha _{1})(1-\alpha _{3})}
		\right) d_{M}(\tilde{x},\tilde{y})}	\end{equation*}%
	 for all $\tilde{x},\tilde{y}\in M$, if $d_{M}(\tilde{x},\tilde{y})>1$, and 
	\begin{equation*}
	r_{s}^{M}(\tilde{x},\tilde{y})\leq cP_{\epsilon +\delta (\Gamma )}(x,y)%
	\begin{cases}
	d_{M}(\tilde{x},\tilde{y})^{2s-n},\quad &\text{if}\quad s<n/2, \\ 
	-\log {d_{M}(\tilde{x},\tilde{y})},\quad &\text{if}\quad s=n/2, \\ 
	1,\quad &\text{if}\quad s>n/2,%
	\end{cases}
	\end{equation*}
	for all $\tilde{x},\tilde{y}\in M$, if $d_{M}(\tilde{x},\tilde{y})\leq 1.$
\end{remark}

\subsection{Functions of the Laplacian}

\label{Functions of the Laplacian}

In this section we apply the estimates of the derivatives of the heat
kernel and we obtain the $L^{p}$-boundedness of some operators related to
the heat semigroup.

\subsubsection{Proof of Theorem \protect\ref{maximal}}

We shall consider separately the \textit{small time} operator
\begin{equation*}
{H}_{\sigma }^{0}(f)(x):=\left( \int_{0}^{1 }e^{2\sigma t}\left(t^{2i-1} \left\vert
\frac{\partial ^{i}}{\partial t^{i}}H_{t}f(x)\right\vert ^{2}+\|\nabla
_{x}H_{t}f(x)\|^{2}\right)   \right) ^{1/2},
\end{equation*}
and the \textit{large time} operator
\begin{equation*}
{H}_{\sigma }^{\infty}(f)(x):=\left( \int_{1}^{\infty }e^{2\sigma t}\left( t^{2i-1}\left\vert
\frac{\partial ^{i}}{\partial t^{i}}H_{t}f(x)\right\vert ^{2}+\|\nabla
_{x}H_{t}f(x)\|^{2}\right)   \right) ^{1/2}.
\end{equation*}
As noted in \cite[p.276]{AN1}, the whole problem comes from the component $H_{\sigma}^{\infty}.$

Note that if $f\in C_0^{\infty}(M)$, then $H_t f=H_t^X f$, where $H_t^X$ is the heat semigroup on $X$, see for example \cite[Lemma 5]{LOMAjga}.

Let
\[
k_{\sigma}^{\infty}(x)=\left( \int_{1}^{\infty }e^{2\sigma t}\left( t^{2i-1}\left\vert
\frac{\partial ^{i}}{\partial t^{i}}h_{t}(x)\right\vert ^{2}+\|\nabla
_{x}h_{t}(x)\|^{2}\right)   \right) ^{1/2}.
\]

Then, the component $H_{\sigma}^{\infty }$ can be handled by estimating
\begin{equation}\label{convinfty}
H_{\sigma}^{\infty }(f)\leq |f| \ast  k_{\sigma}^{\infty}
\end{equation}
and applying the Kunze-Stein phenomenon.

To be more precise, Theorem \ref{Main result} and the gradient estimates of Corollary \ref{Gradient estimate} imply the following upper bound.

\begin{lemma}
	\label{global_lemma} For all $\epsilon\in(0,1)$, there exists $c>0$ such that 
	\begin{equation*}
	\left\vert {k}_{\sigma }^{\infty }(\exp {H})\right\vert \leq ce^{-(1-\epsilon )\langle \rho
		,H\rangle}e^{-(1-\epsilon )\left\Vert H\right\Vert \sqrt{\left\Vert \rho
			\right\Vert ^{2}-\frac{\sigma }{1-\epsilon }}},
	\end{equation*}%
	for all $H\in \overline{\mathfrak{a}^{+}}.$
\end{lemma}

According to Kunze-Stein phenomenon \eqref{kunzestein} for locally symmetric spaces, it holds
	\begin{equation}
	\Vert H_{\sigma }^{\infty }\Vert _{L^{p}(M)\rightarrow L^{p}(M)}\leq
	c\int_{G}|\kappa _{\sigma }^{\infty }(g)|\phi _{-i\eta _{\Gamma }}(g)^{s(p)}{%
		dg},  \label{KS section 5}
	\end{equation} Thus, to prove Theorem \ref{maximal}, it is enough to show that the integral in (\ref{KS section 5}) converges.

	Using the estimates of ${k}_{\sigma }^{\infty }$ obtained in Lemma \ref%
	{global_lemma} and (\ref{spherical estimate}), we get that 
	\begin{align}
	& \int_{G}|{k}_{\sigma }^{\infty }(g)|\phi _{-i\eta _{\Gamma }}(g)^{s(p)}dg 
	\notag \\
	& \leq c\int_{\overline{\mathfrak{a}^{+}}}e^{-\left( 1-\epsilon \right)
		\left( \langle \rho, H \rangle+\sqrt{\left\Vert \rho \right\Vert ^{2}-\frac{\sigma }{%
				1-\epsilon }}\left\Vert H\right\Vert \right) }e^{s(p)(\langle\eta _{\Gamma
		},H\rangle-\langle \rho, H \rangle)}e^{2\langle \rho, H \rangle}{dH}  \notag \\
	& \leq c\underset{\overline{\mathfrak{a}^{+}}}{\int }e^{(1+\epsilon
		-s(p))\left\Vert \rho \right\Vert \left\Vert H\right\Vert +s(p)\left\Vert
		\eta _{\Gamma }\right\Vert \left\Vert H\right\Vert }e^{-(1-\epsilon
		)\left\Vert H\right\Vert \sqrt{\left\Vert \rho \right\Vert ^{2}-\frac{\sigma 
			}{1-\epsilon }}}{dH}.  \label{intoo}
	\end{align}%
	The integral above converges provided that 
	\begin{equation}
	(1+\epsilon -s(p))\left\Vert \rho \right\Vert +s(p)\left\Vert \eta _{\Gamma
	}\right\Vert -(1-\epsilon )\sqrt{\left\Vert \rho \right\Vert ^{2}-\frac{%
			\sigma }{1-\epsilon }}<0.  \label{sigmaepsilon}
	\end{equation}%
	Choosing $\epsilon $ small enough, it follows from (\ref{sigmaepsilon}) that
	the integral in (\ref{intoo}) converges when 
	\begin{equation}
	\sigma <s(p)(\left\Vert \rho \right\Vert -\left\Vert \eta _{\Gamma
	}\right\Vert )(2\left\Vert \rho \right\Vert -s(p)(\left\Vert \rho
	\right\Vert -\left\Vert \eta _{\Gamma }\right\Vert )).  \label{sigma}
	\end{equation}%
	Thus, $H_{\sigma}^{\infty }$ is bounded on $L^{p}(M),$ $p\in (1,\infty ),$
	if (\ref{sigma}) holds true.

	Next, it is left to show that the component $H_{\sigma}^{0}$ is bounded on $L^{p}(M),$ $p\in (1,\infty ).$ We split the operator $H_{\sigma}^{0}$ into two parts using a smooth cut-off function $\psi \in C_{c}^{\infty }(K\backslash G/K)$, with $\psi \equiv 1$ near the origin and $\psi \equiv 0$ in $B\left( 0,2\right) ^{c}$. Then, let ${H}_{\sigma }^{0,0}$ and ${H}_{\sigma}^{0,\infty}$ be the part of the operator associated to $\ast(\psi h_t)$ and  $\ast(1-\psi)h_t$, respectively. 
	
	We observe that the operator $H_{\sigma
}^{0,\infty }$ can be handled like $H_{\sigma}^{\infty }$, and the term ${H}_{\sigma }^{0,0}$ can be handled as in the Euclidean case (see for example \cite[p.278]{AN1}). In particular, Anker \cite[p. 278]{AN1} proves that $H_{\sigma }^{0,0}$
is bounded on $L^{p}(X)$ for all $p\in (1,\infty )$, by controlling $H_{\sigma }^{0,0}$ by a convolution operator that fits in singular integral theory. The same
arguments give the continuity of $H_{\sigma }^{0,0}$ on $L^{p}(G).$ Proceeding as in \cite[Proposition 13]{LOMAann}, we obtain the $L^p$ boundedness of $H_{\sigma }^{0,0}$ on $M$.

\begin{remark}
	Consider the Poisson operator $P_{t}=e^{-t(-\Delta )^{1/2}},$ whose kernel
	is given by 
	\begin{equation*}
	p_{t}=\frac{1}{2\sqrt{\pi }}\int_{0}^{\infty }u^{-3/2}e^{-t^{2}/4u}h_{u}du,
	\end{equation*}%
	(see \cite{AN1} for more details).
	
	Define the corresponding 
	Littlewood-Paley-Stein operators. Then, in a similar way, one can prove that
	these operators are bounded on $L^{p}(M),$ for $M$ as in Theorem \ref{sigmacondition0}, provided that 
	\begin{equation}
	\sigma <\sqrt{s(p)\left( \left\Vert \rho \right\Vert -\left\Vert \eta
		_{\Gamma }\right\Vert \right) \left( 2\left\Vert \rho \right\Vert
		-s(p)\left( \left\Vert \rho \right\Vert -\left\Vert \eta _{\Gamma
		}\right\Vert \right) \right) }.  \label{condition1}
	\end{equation}%
	If $\eta _{\Gamma }=0$ then the condition (\ref{condition1}) on $\sigma $
	becomes $\sigma <2\left\Vert \rho \right\Vert /\sqrt{pp^{\prime }},$ thus we
	recover the result of Anker in \cite{AN1}.
\end{remark}

\subsubsection{Proof of Theorem \protect\ref{Delta}}

In this section we prove Theorem \ref{Delta}, which gives estimates of the
norm
\begin{equation*}
\Vert \Delta _{M}e^{-t\Delta _{M}}\Vert _{L^{p}(M)\rightarrow L^{p}(M)}.
\end{equation*}

From Theorem \ref{Main result}, it follows that for sufficiently small
$\epsilon >0$, there exists $c>0$ such that 
\begin{equation}
\left\vert \frac{\partial h_{t}}{\partial t}\left( \exp H\right) \right\vert
\leq ce^{-\epsilon t}e^{-(1-\epsilon )(\langle \rho ,H\rangle )+\sqrt{%
		\left\Vert \rho \right\Vert ^{2}-\epsilon }\left\Vert H\right\Vert )},\text{
	for }t\geq 1,\text{ }H\in \overline{\mathfrak{a}_{+}}.  \label{global_lemma2}
\end{equation}

Thus, for $t\geq 1,$ proceeding as in the proof of Theorem \ref{maximal}, the
estimate (\ref{global_lemma2}) and the Kunze-Stein phenomenon, imply that 
\begin{equation}
\Vert S_{t}\Vert _{L^{p}(M)\rightarrow L^{p}(M)}\leq ce^{-\epsilon t}\int_{%
	\overline{\mathfrak{a}^{+}}}e^{\left( (1-s(p)+\epsilon )\left\Vert \rho
	\right\Vert -(1-\epsilon )\sqrt{\left\Vert \rho \right\Vert ^{2}-\epsilon }+{%
		s(p)\left\Vert \eta _{\Gamma }\right\Vert }\right) \left\Vert H\right\Vert
}dH.  \label{integralconv}
\end{equation}%
The integral above converges whenever 
\begin{equation*}
s(p)(\left\Vert \eta _{\Gamma }\right\Vert -\left\Vert \rho \right\Vert
)+\epsilon (\left\Vert \rho \right\Vert +\sqrt{\left\Vert \rho \right\Vert
	^{2}-\epsilon })+(\left\Vert \rho \right\Vert -\sqrt{\left\Vert \rho
	\right\Vert ^{2}-\epsilon })<0,
\end{equation*}
which holds true for sufficiently small $\epsilon >0$, since $\left\Vert \eta
_{\Gamma }\right\Vert <\left\Vert \rho \right\Vert .$

Furthermore, from (\ref{integralconv}) we get that 
\begin{equation*}
\Vert S_{t}\Vert _{L^{p}(M)\rightarrow L^{p}(M)}\leq ce^{-\epsilon t},\text{
	for all }t\geq 1.
\end{equation*}

Next, it is left to show that for $t<1$ the operator $S_{t}$ is bounded on $L^{p}(M),$ $p\in [1,\infty].$ We split the operator $S_{t}$ into two parts using a smooth cut-off function $\psi \in C_{c}^{\infty }(K\backslash G/K)$, with $\psi \equiv 1$ near the origin and $\psi \equiv 0$ in $B\left( 0,2\right) ^{c}$. Let ${S}_{t}^{0}=\ast(\psi \frac{\partial h_{t}}{\partial t})$ and ${S}_{t}^{\infty}=\ast(1-\psi)\frac{\partial h_{t}}{\partial t}.$ 

First, note that the operator $S_{t}^{\infty }$ can be handled like $S_{t}$ for $t>1$. Next, it can be shown that the kernel of ${S}_{t}^{0}$ is in $L^1(M)$, using a summation argument and working as in the Euclidean case. Indeed, we have 
	\[
	S_t^0f(x)=\int_X (\psi \frac{\partial h_t}{\partial t})(x,y)f(y)dy=\int_{\Gamma \backslash X} \sum_{\gamma \in \Gamma } (\psi \frac{\partial h_t}{\partial t})(x,\gamma y)f(\gamma y)d\tilde{y}.
	\]
	Using Theorem \ref{Main result} and the fact that $\psi$ is supported around the origin,
	\begin{align*}
	\sum_{\gamma \in \Gamma } (\psi \frac{\partial h_t}{\partial t})(x,\gamma y)&\leq c\sum_{\gamma \in \Gamma:\; d(x,\gamma y)\leq 1 }t^{-n/2-1}e^{-(1-\epsilon)d^2(x,\gamma y)/(4t)}\\
	&\leq c_{\Gamma}t^{-n/2-1}e^{-(1-\epsilon)d_M^2(\tilde{x},\tilde{y})/(4t)},
	\end{align*}
	since the last sum is finite.

\end{document}